\documentclass[pdflatex,sn-mathphys-num]{sn-jnl}

\usepackage{graphicx}%
\usepackage{multirow}%
\usepackage{amsmath,amssymb,amsfonts}%
\usepackage{amsthm}%
\usepackage{mathrsfs}%
\usepackage[title]{appendix}%
\usepackage{xcolor}%
\usepackage{textcomp}%
\usepackage{manyfoot}%
\usepackage{booktabs}%
\usepackage{algorithmicx}%
\usepackage{algpseudocode}%
\usepackage{listings}%
\usepackage{multirow}
\usepackage[ruled,vlined]{algorithm2e}

\def\R{\mathbb{R}}

\def\H2{H^2(\R^N)}
\def\L2{L^2(\R^N)}

\def\H{{\cal H}}
\def\H1{H^1(\R)}

\DeclareMathOperator{\CP}{CP}

\theoremstyle{thmstyleone}%
\newtheorem{theorem}{Theorem}
\theoremstyle{thmstyletwo}%
\newtheorem{example}{Example}%
\newtheorem{lemma}[theorem]{Lemma}
\theoremstyle{thmstylethree}%
\allowdisplaybreaks
\begin{document}

\title[An Efficient Numerical Scheme for a Time-Fractional Burgers Equation with Caputo-Prabhakar Derivative]{An Efficient Numerical Scheme for a Time-Fractional Burgers Equation with Caputo-Prabhakar Derivative}

\author[1]{\fnm{Deeksha} \sur{Singh}}\email{deekshas905@gmail.com}
\equalcont{These authors contributed equally to this work.}

\author*[2]{\fnm{Swati} \sur{Yadav}}\email{swati@iitbhilai.ac.in}
\equalcont{These authors contributed equally to this work.}

\author[1]{\fnm{Rajesh Kumar} \sur{Pandey}}\email{rkpandey.mat@iitbhu.ac.in}
\equalcont{These authors contributed equally to this work.}

\affil[1]{\orgdiv{Department of Mathematical Sciences}, \orgname{Indian Institute of Technology (BHU) Varanasi}, \orgaddress{\postcode{221005}, \state{Uttar Pradesh}, \country{India}}}

\affil*[2]{\orgdiv{Department of Mathematics}, \orgname{Indian Institute of Technology Bhilai}, \orgaddress{\city{Durg}, \postcode{491002}, \state{Chhattisgarh}, \country{India}}}


\abstract{This paper presents a numerical method to solve a time-fractional Burgers equation, achieving order of convergence $(2-\alpha)$ in time, here $\alpha$ represents the order of the time derivative. The fractional derivative is modeled by the Caputo–Prabhakar (CP) formulation, which incorporates a kernel defined by the three-parameter Mittag-Leffler function. Finite difference methods are employed for the discretization of the derivatives. To handle the non-linear term, the Newton iteration method is used. The proposed numerical scheme has been shown to be stable and convergent in the norm $L_{\infty}$. The validity of the theory is supported by two numerical examples. }




\maketitle

\section{Introduction}\label{sec1}

Fractional calculus has proven to be a powerful tool for modeling complex physical phenomena, including viscoelasticity, mechanics, and heat transfer in materials with memory. Caputo and Cametti \cite{caputo2008diffusion} introduced a modified model for transport processes in membranes using fractional derivatives to account for solid particles. Laskin \cite{laskin2002fractional} extended the study to the fractional Schrödinger equation, providing insights into the energy spectra of hydrogen-like atoms. A fractional complex-order model has been applied by Pinto and Carvalho \cite{pinto2016fractional} to study drug resistance in HIV infection. Additionally, Caputo and Fabrizio \cite{caputo2015damage} employed fractional derivatives to describe damage and fatigue in viscoelastic solids. For a comprehensive exploration of fractional calculus and its applications, readers are encouraged to refer to the seminal works by Kilbas et al. \cite{kilbas2006theory} and Podlubny \cite{podlubny1998fractional}.
	
Unlike the classical definition of integer-order derivatives, fractional derivatives are defined in multiple ways, each suited to different contexts. Various perspectives have been presented by Caputo, Hadamard, Gr{\"u}nwald, Letnikov, Riemann, and Riesz offering solutions to a wide range of real-world problems \cite{ahmad2017hadamard,garrappa2016grunwald,kilbas2006theory,muslih2010riesz,podlubny1998fractional}. 
In this paper, we consider the Caputo--Prabhakar (CP) derivative which is a generalization of the Caputo derivative obtained by introducing three parameter  Mittag--Leffler function, { termed as \emph{Prabhakar function}, in the kernel \cite{prabhakar1969two}. For an absolutely continuous function  $f \in \text{AC}^1(0,L) $ of order $\alpha \in (0,1)$, we can define CP derivative as} 
    \begin{align}\label{eq:1}    ^{\CP}_{~~0}D_{\rho,\alpha,\omega}^{\gamma}f(t)=\int_{0}^{t}(t-s)^{-\alpha}E_{\rho,1-\alpha}^{-\gamma}(\omega(t-s)^{\rho})f'(s)ds,
    \end{align}
 where 
	\begin{align*}
		E_{\alpha,\beta}^{\gamma}(\zeta)=\sum\limits_{m=0}^{\infty}\frac{\Gamma(\gamma+m) \zeta^{m}}{\Gamma (\gamma) \Gamma(\alpha m+\beta) m!},
	\end{align*}
 is a three parameter Mittag-Leffler function, here  $\alpha,\beta,\gamma$ are complex numbers. A Prabhakar function $E_{\alpha,\beta}^{\gamma}$ is an entire function of order $(\text{Re}(\alpha))^{-1}$ provided $\text{Re}(\alpha) > 0$ and it reduces to many very well-known special functions for particular values of $\alpha,\beta,\gamma$. 
Subsequently, the integral \eqref{eq:1} exists for $Re(\alpha)>0$. For convenience, we denote the CP derivative as 
\begin{align*}
    ^{\CP}_{~~0}D_{\rho,\alpha,\omega}^{\gamma} (\cdot) := 
\frac{^{\CP}\partial^{\alpha}(\cdot)}{\partial t^{\alpha}}.
\end{align*}

There are also real-world applications of Prabhakar fractional derivatives; CP derivative is used in describing the fractional Poisson process and dielectric relaxation phenomena \cite{garrappa2016grunwald, garrappa2016models}, for instance. Moreover, it describes a fractional Maxwell model in the linear viscoelasticity \cite{giusti2018prabhakar} and a mathematical modeling of fractional differential filtration dynamics \cite{bulavatsky2017mathematical}. 
Some other applictions of Prabhakar fractional derivatives can be seen in fractional dynamical systems \cite{derakhshan2016asymptotic}; generalized reaction--diffusion equations \cite{agarwal2017analytic}, generalized model of particle deposition in porous media \cite{xu2017time}.
Properties and applications of the Prabhakar function have been studied by Garra et al.\ in \cite{GARRA2018314} with a special attention to the asymptotic behavior of the function in the complex plane and on the negative semi-axis. Mainardi et al.\ \cite{mainardi2015complete} have examined the local integrability and complete monotonicity of the Prabhakar function. Srivastava et al.\ \cite{srivastava2019some} established some new connections between the ML functions of one, two, and three parameters and expressed the three-parameter ML function as a fractional derivative of the two-parameter ML function. For a more detailed study of the Prabhakar function and derivatives and integrals involving the Prabhakar function, we refer readers to \cite{giusti2020general, giusti2020practical} and the references therein.

Numerical approximation of CP derivatives have been used to find approximate solutions of fractional-order integro-differential equations using different approaches such as Legendre wavelet method \cite{abbaszadeh2021solving}, Haar wavelet collocation method \cite{marasi2022haar}, and Chebyshev polynomials \cite{bagharzadehtvasani2022numerical}. Derakhshan et al.\ studied numerical approximation of fractional Sturm--Liouville problems using CP derivatives \cite{derakhshan2019numerical}.
Numerous studies have been done on numerical solutions of fractional-order differential equations with CP derivatives. For example, Eshaghi et al.\ studied a fractional Black--Scholes model \cite{eshaghi2017fractional} and stability and dynamics of neutral and integro--differential systems of fractional order \cite{eshaghi2020stability}. Derakhshan et al.\ gave a comparison between homotopy perturbation transform and fractional Adams--Bashforth method for the solution of nonlinear fractional differential equations \cite{derakhshan2020comparison}, Garrappa et al.\ studied stability of fractional-order systems with Caputo--Prabhakar derivative \cite{garrappa2020stability}. In \cite{singh2022approximation} Singh et al. discussed numerical approximation of Caputo--Prabhakar derivative using polynomials of degree two and three and their application in solving Advection-Diffusion equation. Extending their work in \cite{singh2024high}, the authors disccused a higher order approximation of Caputo-Prabhakar derivative using polynomials of degree $r$.  

In this paper, we present a second-order convergent numerical scheme to solve the fractional Burgers equation, which describes traveling waves with front sharpening. The numerical approximation is based on a time-stepping scheme for the CP derivative, combined with a central finite difference method for spatial discretization. To achieve second-order convergence and reduce computational costs, Newton's linearization technique is applied to handle the nonlinear term. A rigorous theoretical analysis, including stability and convergence, is provided, demonstrating that the scheme achieves a temporal convergence rate of 
$(2-\alpha)$ and a spatial convergence rate of 2.
\subsection{Mathematical model}

We investigate the following  fractional Burgers equation involving CP derivative in the time direction:
	\begin{align}\label{Main.eq}
		\frac{^{\CP}\partial^{\alpha}u(x,t)}{\partial t^{\alpha}}+u(x,t)\frac{\partial u(x,t)}{\partial x}= \frac{\partial^2u(x,t)}{\partial x^2}+f(x,t), \quad (x,t) \in (0,L) \times (0,T),
	\end{align}
	with initial and boundary conditions
	\begin{eqnarray}\label{eq2}{
			\left.
			\begin{array}{lc@{}c@{}r}
				u(x,0)=\phi(x), \quad x \in (0,L),\\
				u(0,t)=u(L,t)=0,\quad t \in (0,T).
			\end{array}\right\}
	}\end{eqnarray} 
For our case, $\alpha \in (0,1)$, $L$ and $T$ are positive real numbers. $f(x,t)$ is a given source term. The functions $f(x,t)$ and $\phi(x)$ are sufficiently smooth known functions.

In \cite{yadav2020numerical}, the authors addressed a similar problem using the Atangana–Baleanu Caputo derivative by applying a direct linearization technique, which led to a reduced convergence order in the temporal direction. In contrast, we employ the Newton iteration method to achieve a higher convergence rate and improved accuracy.
It is important to note that the scheme presented here is a novel contribution, as its application to solve problem \eqref{Main.eq} has not been previously explored in the literature.

The remainder of this paper is organized as follows: In Section \ref{Sec:2}, we derive a numerical scheme for solving the time-fractional Burgers equation, followed by a rigorous analysis of stability and convergence. Section \ref{Sec:3} provides numerical examples that validate the sharpness of the achieved convergence rates. Finally, some concluding remarks are drawn in Section \ref{sec:4}.

\section{Derivation and Analysis of the Numerical Scheme}\label{Sec:2}

We assume the grid points as $ t_k=k\tau, ~k=0,1,\dots,N$ in time direction and for space direction, $x_i=ih, ~i=0,1,\dots,M$, where $ \tau=\frac{T}{N}$ and $ h=\frac{L}{M},$ for $M, N \in \mathbb{Z}^+ $. 
For the discretization of time-derivative of \eqref{Main.eq}, we use finite difference method (FDM) in the definition \eqref{eq:1}. So at a grid point $(x_i,t_k)$, we have
\begin{align}\label{Time-approx}
	\frac{^{\CP}\partial^{\alpha}u}{\partial t^{\alpha}}&\vert_{(x_i,t_k)} = \int^{t_k}_{0}(t_{k}-s)^{-\alpha}E_{\rho, 1-\alpha}^{-\gamma}(\omega(t_{k}-s)^{\rho})u'(x_i,s)ds \nonumber \\
	\approx & \sum_{j=0}^{k-1} \int^{t_{j+1}}_{t_j}  (t_{k}-s)^{-\alpha}E_{\rho, 1-\alpha}^{-\gamma}(\omega(t_{k}-s)^{\rho})u'(x_i,t_j)ds \nonumber\\
	= &  \frac{u(x_i,t_{j+1})-u(x_i,t_j)}{\tau} \sum_{j=0}^{k-1} \int^{t_{j+1}}_{t_j} (t_{k}-s)^{-\alpha} E_{\rho, 1-\alpha}^{-\gamma}(\omega(t_{k}-s)^{\rho})ds+ r_k \nonumber\\
	= & \frac{1}{\mu}[a_1u_i^k - \sum_{j=1}^{k-1}(a_{k-j}-a_{k-j+1})u_i^j-a_ku_i^0]+ r_k,
	\end{align}
where $r_k$ is the truncation error. Eq. \eqref{Time-approx} is obtained after integration and rearrangement of the terms, where we have introduced some notations as  $\mu=\tau^\alpha,\ u_i^k = u(x_i,t_k),$ and 
	\begin{align}\label{a-n}
		a_n= n^{1-\alpha}E_{\rho, 2-\alpha}^{-\gamma}(\omega(t_{n})^{\rho})-(n-1)^{1-\alpha}E_{\rho, 2-\alpha}^{-\gamma}(\omega(t_{n-1})^{\rho}),
	\end{align}
for $n=1,2,\dots,N$.

Following lemma gives a bound on the truncation error $r_k$ in the approximation \eqref{Time-approx} provided the solution $u$ has a continuous second order derivative. Proof of the lemma can be found in \cite{singh2022approximation}.
 \begin{lemma}\cite{singh2022approximation} 
 Let $ u(.,t)\in C^{2}[0,1] $ then for any $\alpha \in (0,1) $, the truncation error $r_k$ occurred in approximating the CP derivative using FDM, given by   
\begin{align*}
	{r_k} & = \frac{^{\CP}\partial^{\alpha}u(x,t)}{\partial t^{\alpha}}\vert_{(x_i,t_k)}-\frac{1}{\mu}[a_1u_i^k-\sum_{j=1}^{k-1}(a_{k-j}-a_{k-j+1})u_i^j-a_ku_i^0],
\end{align*} 
is bounded as 
\begin{align}
	\lvert r_k\rvert &\leq C_1\frac{\tau^{2-\alpha}}{2}\max_{0\leq t\leq t_{k-1}}u''(x_i,t_k),\ 
\end{align} 
where $ C_1$ is a constant.
\end{lemma} 
We approximate the space derivatives of integer order from \eqref{Main.eq} as following,
\begin{align}
	\frac{\partial^{2}u(x,t)}{\partial x^{2}}\vert_{(x_i,t_k)} & \approx \frac{u^k_{i+1}-2u^k_{i}+u^k_{i-1}}{h^2}
 := \delta^2_x u_i^k,\label{space-app1}\\
	\frac{\partial u(x,t)}{\partial x}\vert_{(x_i,t_k)}
    & \approx \frac{u^k_{i+1}-u^k_{i-1}}{2h}
    := \frac{1}{2h}\Delta u_i^k,\label{space-app2}\\
	u \frac{\partial u(x,t)}{\partial x}\vert_{(x_i,t_k)}&
	\approx \frac{u_{i+1}^k+u_{i}^k+u_{i-1}^k}{3}\frac{u^k_{i+1}-u^k_{i-1}}{2h}
 \nonumber\\
	& =\frac{1}{6h}(u_i^k \Delta u_i^k+\Delta( u_i^k)^2)\label{space-app3}.
\end{align}
The discretizations \eqref{space-app1} and \eqref{space-app2} are simply central differences, while  (\ref{space-app3}) is based on the Galerkin method using piecewise linear test functions \cite{lopez1990difference,qiu2019implicit}.
To get the fully discrete numerical scheme for the fractional Burgers equation (\ref{Main.eq}), we substitute the approximations at point $ (x_i,t_k) $ from \eqref{Time-approx}, \eqref{space-app1}, and \eqref{space-app3},
	\begin{align}\label{main-scheme1}
		\frac{1}{\mu}[a_1u_i^k-\sum_{j=1}^{k-1}(a_{k-j}-a_{k-j+1})u_i^j-a_ku_i^0]+ \frac{1}{6h}(u_i^k \Delta u_i^k+\Delta( u_i^k)^2)= \delta^2_x u_i^k+f_i^k, 
	\end{align}
for $1\leq i \leq M,\ 1\leq k \leq N$, where $f_i^k=f(x_i,t_k)$. The initial and boundary conditions are as
\begin{align}\label{initian-cond}
	u^k_0=u^k_M=0,~~~~u^0_i=\phi(x_i).
\end{align}

\subsection{Stability and Convergence}\label{subsection:2.2}

First, we introduce some notations directly inspired by \cite{lopez1990difference,qiu2019implicit}, and we recommend these papers for a detailed explanation.
Define $ W=\{w=(w_0,w_1,\dots w_M)\vert w_0=w_M=0\}, $ and for any two elements $v,w \in W$ define inner product as 
$\langle v,w \rangle_h=\sum_{i=1}^{M-1}hv_{i}w_{i}$.
The corresponding norms are
$\|w\|_h=\sqrt{\langle w,w \rangle_h},~~~~ \|w\|_{\infty}= \max_{1\leq i\leq {M-1}} \vert w_i \vert.$
 %
Given the boundary conditions by \eqref{initian-cond}, we have
$ \|w\|_{\infty}\leq \frac{\sqrt{L}}{2}\|\delta_xw\|_h, $ where $ \delta_xw_i=\frac{1}{h}(w_{i}-w_{i-i}) .$


%
For the space $W$ following results hold true
	\begin{align}
		\langle \delta_x^2 w_1,w_2 \rangle_h= -\langle \delta_x w_1,\delta_x w_2 \rangle_h ~~~~ \forall w_1,w_2 \in W,\\
		\langle w\Delta w+\Delta (w)^2,w\rangle_h =0, ~~~~ \forall w \in W,\label{Eq2.23}
	\end{align}
 where $\delta_x^2$ and $\Delta$ are given by \eqref{space-app1} and \eqref{space-app2}, respectively.
 \begin{lemma}
	For any vector solution $ U = (U_0,U_1,U_2,\dots,U_N)$ to \eqref{main-scheme1}-\eqref{initian-cond}, the following inequality is satisfied
	\begin{align}\label{Ineq_Lem_2}
		\sum_{k=1}^{N} & \left[ -a_kU_0 -\sum_{j=1}^{k-1}(a_{k-j}-a_{k-j+1})U_j+a_1U_k \right] U_k \nonumber \\
		&\geq \frac{1}{2}\left\{N^{1-\alpha}E_{\rho, 2-\alpha}^{-\gamma}(\omega(t_{N})^{\rho})-(N-1)^{1-\alpha}E_{\rho, 2-\alpha}^{-\gamma}(\omega(t_{N-1})^{\rho})\right\} \sum_{k=1}^{N}U_k^2\nonumber\\
		&-\frac{N^{1-\alpha}}{2}E_{\rho, 2-\alpha}^{-\gamma}(\omega(t_{N})^{\rho})U_0^2,
	\end{align}
	where $ \{a_n\vert n=1,\dots,k\} $ are given by Eq. (\ref{a-n}), and are decreasing function of $ n $.
\end{lemma}
\begin{proof} { By taking the left hand side of the inequality \eqref{Ineq_Lem_2}, we can simplify it as follows using simple well-known relations and inequalities}
\begin{align*}
	&\sum_{k=1}^{N}  \left[-a_kU_0 -\sum_{j=1}^{k-1}(a_{k-j}-a_{k-j+1})U_j+a_1U_k \right] U_k\nonumber\\
	&= -\sum_{k=1}^{N} a_k(U_0 U_k)- \sum_{k=2}^{N} \sum_{j=1}^{k-1} (a_{k-j}-a_{k-j+1}) U_j U_k+ a_1 \sum_{k=1}^{N} U_k^2 \nonumber\\
	&\geq -\frac{1}{2} \sum_{k=1}^{N} a_k (U_0^2+U_k^2)-\frac{1}{2} \sum_{k=2}^{N} \sum_{j=1}^{k-1} (a_{k-j}-a_{k-j+1}) (U_j^2+U_k^2) + a_1 \sum_{k=1}^{N} U_k^2 \nonumber\\
	&= -\frac{1}{2} \sum_{k=1}^{N} a_k (U_0^2+U_k^2)-\frac{1}{2} \sum_{k=1}^{N} (a_{1}-a_{k}) U_k^2 -\frac{1}{2} \sum_{k=1}^{N-1} (a_{1}-a_{N-k+1}) U_k^2+ a_1 \sum_{k=1}^{N} U_k^2\nonumber\\
	&= -\frac{1}{2} \sum_{k=1}^{N} a_k U_0^2+ \sum_{k=1}^{N} (\frac{1}{2}a_{N-k+1}) U_k^2
        \geq -\frac{1}{2} \sum_{k=1}^{N} a_k U_0^2+ \frac{1}{2}a_{N}\sum_{k=1}^{N} U_k^2.
\end{align*}
 Also, since $\sum_{k=1}^{N} a_k= N^{1-\alpha} E_{\rho, 2-\alpha}^{-\gamma}(\omega (t_{N})^{\rho})$ and\\
 $a_{N} = N^{1-\alpha}E_{\rho, 2-\alpha}^{-\gamma}(\omega (t_{N})^{\rho})-(N-1)^{1-\alpha}E_{\rho, 2-\alpha}^{-\gamma}(\omega (t_{N-1})^{\rho})  $, we have the inequality \eqref{Ineq_Lem_2}.
 \end{proof}

Following theorem proves the boundedness of the numerical solution obtained from the proposed numerical scheme (\ref{main-scheme1}) which shows the scheme \eqref{main-scheme1} is stable.

\begin{theorem}\label{Th2.1}
    Let the solution of the proposed numerical scheme (\ref{main-scheme1}) with boundary and initial conditions (\ref{initian-cond}) is $ \{u_j^k \vert 1\leq j \leq M-1,1\leq k \leq N\} $, then we have
\begin{align}
	\sum_{k=1}^{N} \|u^k\|^2_{\infty} \leq &\frac{N^{1-\alpha}L}{8\mu} E_{\rho, 2-\alpha}^{-\gamma}(\omega t_{N}^{\rho}) \|u^0\|_h^2 \nonumber \\&+\frac{\mu L}{8\left\{N^{1-\alpha}E_{\rho, 2-\alpha}^{-\gamma}(\omega (t_{N})^{\rho})-(N-1)^{1-\alpha}E_{\rho, 2-\alpha}^{-\gamma}(\omega (t_{N-1})^{\rho}) \right\}}\|f^k\|_h^2.\nonumber
\end{align}	
\end{theorem}
\begin{proof} 
    By taking inner product of Eq. (\ref{main-scheme1}) with $ u^k $ then by summing up from 1 to $ N $, we get
\begin{align*}
	\frac{1}{\mu} \sum_{k=1}^{N} \langle a_1u^k-\sum_{j=1}^{k-1}(a_{k-j}
    - a_{k-j+1})u^j-a_ku^0, u^k\rangle_h 
    & + \frac{1}{6h}\sum_{k=1}^{N}\langle u^k \Delta u^k+\Delta( u^k)^2, u^k\rangle_h\\
    = \sum_{k=1}^{N}\langle f^k,u^k \rangle_h   & - \sum_{k=1}^{N} \| \delta_x u^k \|_h^2,
\end{align*}
\begin{align*}
	\mu \sum_{k=1}^{N} \| \delta_x u^k \|_h^2  & \leq
	\mu \sum_{k=1}^{N}\langle f^k,u^k \rangle_h  \frac{N^{1-\alpha}}{2} E_{\rho, 2-\alpha}^{-\gamma}(\omega (t_{N})^{\rho}) \|U_0\|_h^{2} \\
    - \frac{1}{2} & \left\{N^{1-\alpha}E_{\rho, 2-\alpha}^{-\gamma}(\omega (t_{N})^{\rho})-(N-1)^{1-\alpha}E_{\rho, 2-\alpha}^{-\gamma}(\omega (t_{N-1})^{\rho}) \right\} \sum_{k=1}^{N} U_k^2.
\end{align*}
In the last inequality above, we used Eq.(\ref{Ineq_Lem_2}) and Eq.(\ref{Eq2.23}). By the virtue of Young's inequality: for all $a, b \in \mathbb{R}$ there exists $\varepsilon>0$ such that $ ab\leq \varepsilon a^2+\frac{1}{4\varepsilon }b^2,$ we have
\begin{align*}
    \mu \sum_{k=1}^{N} \| \delta_x & u^k \|_h^2 \leq \\
    - & \frac{1}{2} \left\{  N^{1-\alpha}  E_{\rho, 2-\alpha}^{-\gamma}(\omega (t_{N})^{\rho})
    - (N-1)^{1-\alpha} E_{\rho, 2-\alpha}^{-\gamma}(\omega (t_{N-1})^{\rho})  \right\} \sum_{k=1}^{N} \|u^k\|_h^2 \\
    + & \frac{N^{1-\alpha}}{2} E_{\rho, 2-\alpha}^{-\gamma}(\omega(t_{N})^{\rho}) \|u^0\|_h^2  \\
    + & \ \mu \sum_{k=1}^{N} \left(\frac{\left\{N^{1-\alpha}E_{\rho, 2-\alpha}^{-\gamma}(\omega (t_{N})^{\rho})-(N-1)^{1-\alpha}E_{\rho, 2-\alpha}^{-\gamma}(\omega (t_{N-1})^{\rho}) \right\}}{2\mu}\|u^k\|_h^2 \right.  \\
    + & \left. \frac{\mu}{2\{N^{1-\alpha}E_{\rho, 2-\alpha}^{-\gamma}(\omega (t_{N})^{\rho})-(N-1)^{1-\alpha}E_{\rho, 2-\alpha}^{-\gamma}(\omega (t_{N-1})^{\rho}) \}}\|f^k\|_h^2\right).
		\end{align*}
Since, $ \|u\|_{\infty}\leq \frac{\sqrt{L}}{2}\|\delta_xu\|_h,$
\begin{align*}
	\frac{4\mu}{L} \sum_{k=1}^{N} & \|u^k\|_{\infty}^2 \leq \frac{N^{1-\alpha}}{2} E_{\rho, 2-\alpha}^{-\gamma} (\omega (t_{N})^{\rho}) \|u^0\|_h^2 \\
	+ & \frac{\mu^{2}}{2\{N^{1-\alpha}E_{\rho, 2-\alpha}^{-\gamma}(\omega (t_{N})^{\rho})-(N-1)^{1-\alpha}E_{\rho, 2-\alpha}^{-\gamma}(\omega (t_{N-1})^{\rho})\}}\sum_{k=1}^{N}\|f^k\|_h^2, \\
	\Longrightarrow ~~ \sum_{k=1}^{N} & \|u^k\|_{\infty}^2 \leq \frac{N^{1-\alpha}L}{8\mu} E_{\rho,2-\alpha}^{-\gamma}(\omega (t_{N})^{\rho})\|u^0\|_h^2 \\
	+ & \frac{\mu L}{8\{N^{1-\alpha}E_{\rho, 2-\alpha}^{-\gamma}(\omega (t_{N})^{\rho})-(N-1)^{1-\alpha}E_{\rho, 2-\alpha}^{-\gamma}(\omega (t_{N-1})^{\rho})\}}\sum_{k=1}^{N}\|f^k\|_h^2. 
\end{align*}
\end{proof}

Now, we will discuss the convergence of our scheme for which we use energy method. Assume $ v(x_i,t_k)=v^k_i $ be the exact solution of the problem (\ref{Main.eq})-(\ref{eq2}), and let $ \epsilon_i^k=v(x_i,t_k)-u^k_i $ be the error between exact solution and the approximated solution, then $ \epsilon_0^k= \epsilon_M^k=0$ and $ \epsilon_i^0= 0 ~~ \forall~~1\leq i \leq M-1,~~1\leq k \leq N $, and
\begin{align*}
	\left[\mu \frac{^{\CP}\partial^{\alpha}v(x_i,t_k)}{\partial t^{\alpha}} - \left(a_1u_i^k- \sum_{j=1}^{k-1}(a_{k-j} - a_{k-j+1})u_i^j-a_ku_i^0\right) \right] & \\
	+ \mu\left[ v(x_i,t_k)\frac{\partial v(x_i,t_k)}{\partial x} - \frac{1}{6h}(u_i^k \Delta u_i^k+\Delta( u_i^k)^2)\right] = & \mu\left[\frac{\partial^2v(x_i,t_k)}{\partial x^2} - \delta^2_x u_i^k\right],
\end{align*}

\begin{align}
	\Rightarrow ~~[a_1\epsilon_i^k & -\sum_{j=1}^{k-1}(a_{k-j}-a_{k-j+1})\epsilon_i^j-a_k\epsilon_i^0]- \frac{\mu}{6h}(\epsilon_i^k \Delta \epsilon_i^k+\Delta( \epsilon_i^k)^2)\nonumber\\
    & = \mu\delta^2_x \epsilon_i^k 	
     -\mu \left[ \frac{^{\CP} \partial^{\alpha} v(x_i,t_k)}{\partial t^{\alpha}} - \frac{1}{\mu}\left(a_1v_i^k- \sum_{j=1}^{k-1}(a_{k-j} - a_{k-j+1})v_i^j-a_kv_i^0\right) \right]\nonumber\\ 
	& + \mu\left[ \frac{\partial^{2} v(x_i,t_k)}{\partial x^{2}} - \delta^2_x v_i^k\right] 
    -\mu \left[v(x_i,t_k)\frac{\partial v(x_i,t_k)}{\partial x}-\frac{1}{6h}(v_i^k \Delta v_i^k+\Delta( v_i^k)^2)\right]\nonumber\\
 &-\frac{\mu }{6h}\left[ \left( v_i^k \Delta\epsilon_i^k+ \Delta(\epsilon_i^k v_i^k)\right) + \epsilon_i^k \Delta v_i^k+\Delta(\epsilon_i^k v_i^k)\right].\label{Eq2.18}
\end{align}

\noindent Let us denote
\begin{align*}
	{(\gamma_1)}_i^k &= \frac{\partial^{2} v(x_i,t_k)}{\partial x^{2}} - \delta^2_x v_i^k,\\
	{(\gamma_2)}_i^k& = \frac{^{\CP}\partial^{\alpha}v(x_i,t_k)}{\partial t^{\alpha}} - \frac{1}{\mu}\left(a_1v_i^k- \sum_{j=1}^{k-1}(a_{k-j} - a_{k-j+1})v_i^j-a_kv_i^0\right),\\
    {(\gamma_3)}_i^k& = v(x_i,t_k)\frac{\partial v(x_i,t_k)}{\partial x}-\frac{1}{6h}(v_i^k \Delta v_i^k+\Delta( v_i^k)^2),\\
    {(\gamma_4)}_i^k &= v_i^k \Delta\epsilon_i^k+ \Delta(\epsilon_i^k v_i^k),\\
	{(\gamma_5)}_i^k &= \epsilon_i^k \Delta v_i^k+\Delta(\epsilon_i^k v_i^k),
	\end{align*}
	and by rewriting Eq. (\ref{Eq2.18}), we have
\begin{align}
	[a_1\epsilon_i^k&-\sum_{j=1}^{k-1}(a_{k-j}-a_{k-j+1})\epsilon_i^j-a_k\epsilon_i^0]- \frac{\mu}{6h}(\epsilon_i^k \Delta \epsilon_i^k+\Delta( \epsilon_i^k)^2)= \mu\delta^2_x \epsilon_i^k\nonumber\\
	&+\mu \left[ {(\gamma_1)}_i^k-{(\gamma_2)}_i^k-{(\gamma_3)}_i^k -\frac{1}{6h}\left({(\gamma_4)}_i^k-{(\gamma_5)}_i^k\right) \right].\label{eq2.19}
	\end{align}
We can easily calculate the following results by using initial and boundary conditions ( see  \cite[Lemma 3.6, Lemma 3.7]{qiu2019implicit})
\begin{align}
	\langle {(\gamma_4)}_i^k, \epsilon^k \rangle_h =0, ~~~~ \vert\langle {(\gamma_5)}_i^k, \epsilon^k\rangle_h\vert \leq 3C_0h\|\epsilon^k\|_h^2,\\
	\sum_{k=1}^N \tau(\|{(\gamma_1)}_i^k\|_h^2+\|{(\gamma_2)}_i^k\|_h^2 +\|{(\gamma_3)}_i^k\|_h^2) \leq C(\tau^{2-{\alpha}}-h^2)^2, \label{Eq2.29}
\end{align}
 where $C_0$ and $C$ are real positive constants.

Now, in Eq. (\ref{eq2.19}), taking product with $ h\epsilon_i^k $ and then  summing up for $ i $ from 1 to $ M-1 $, 
\begin{align*}
	\langle a_1\epsilon^k & -\sum_{j=1}^{k-1}(a_{k-j}-a_{k-j+1})\epsilon^j -a_k\epsilon^0,\epsilon^k\rangle_h - \frac{\mu}{6h}\langle\epsilon^k \Delta \epsilon^k + \Delta( \epsilon^k)^2,\epsilon^k\rangle_h \nonumber\\
    & = \mu\langle\delta^2_x \epsilon^k,\epsilon^k \rangle_h
	+\mu \langle {(\gamma_1)}^k-{(\gamma_2)}^k-{(\gamma_3)}^k -\frac{1}{6h}\left({(\gamma_4)}^k-{(\gamma_5)}^k\right),\epsilon^k\rangle_h, 
\end{align*}
\begin{align*}
	& \Rightarrow ~ \sum_{k=1}^N\mu \| \delta_x \epsilon^k\|_h^2 = - \sum_{k=1}^N \langle a_1 \epsilon^k - \sum_{j=1}^{k-1}(a_{k-j}-a_{k-j+1}) \epsilon^j - a_k\epsilon^0, \epsilon^k \rangle_h ~~~~~~~~~ \\
    &~~~~~~~~~~~~~~~~  + \frac{\mu}{6h}\sum_{k=1}^N\langle\epsilon^k \Delta \epsilon^k + \Delta( \epsilon^k)^2,\epsilon^k\rangle_h \\
    &~~~~~~~~~~~~~~~~ + \mu\sum_{k=1}^N \langle {(\gamma_1)}^k-{(\gamma_2)}^k-{(\gamma_3)}^k -\frac{1}{6h}\left({(\gamma_4)}^k-{(\gamma_5)}^k\right),\epsilon^k\rangle_h \\
    & \leq - \frac{1}{2} \{N^{1-\alpha}E_{\rho, 2-\alpha}^{-\gamma}(\omega (t_{N})^{\rho})-(N-1)^{1-\alpha}E_{\rho, 2-\alpha}^{-\gamma}(\omega (t_{N-1})^{\rho})\} \sum_{k=1}^{N} \|\epsilon^k\|_h^2\\
    &~~~~~~~~~~~~~~~~ + \frac{N^{1-\alpha}}{2} E_{\rho,2-\alpha}^{-\gamma} (\omega t_{N}^{\rho}) \| \epsilon^0\|_h^2 \\
    &~~~~~~~~~~~~~~~~ + \mu \langle {(\gamma_1)}^k-{(\gamma_2)}^k-{(\gamma_3)}^k -\frac{1}{6h}\left({(\gamma_4)}^k-{(\gamma_5)}^k\right),\epsilon^k\rangle_h\\
	& \leq - \frac{1}{2} \{N^{1-\alpha}E_{\rho, 2-\alpha}^{-\gamma}(\omega (t_{N})^{\rho})-(N-1)^{1-\alpha}E_{\rho, 2-\alpha}^{-\gamma}(\omega (t_{N-1})^{\rho})\} \sum_{k=1}^{N} \|\epsilon^k\|_h^2 \\
    &~~~~~~~~~~~~~~~~ +\mu\sum_{k=1}^N \langle {(\gamma_1)}^k-{(\gamma_2)}^k-{(\gamma_3)}^k ,\epsilon^k\rangle_h + \frac{C_0}{2}\mu \sum_{k=1}^N \|\epsilon^k\|_h^2\\
    & \leq \left(\frac{C_0}{2}- \frac{1}{2\mu } \{N^{1-\alpha}E_{\rho, 2-\alpha}^{-\gamma}(\omega (t_{N})^{\rho})-(N-1)^{1-\alpha}E_{\rho, 2-\alpha}^{-\gamma}(\omega (t_{N-1})^{\rho})\}\right) \sum_{k=1}^{N} \|\epsilon^k\|_h^2 \\
    &~~~~~~~~~~~~~~~~ +\sum_{k=1}^{N}\left( \| {(\gamma_1)}^k\|_h + \|{(\gamma_2)}^k\|_h + \|{(\gamma_3)}^k\|_h \right) \|\epsilon^k\|_h.
\end{align*}
Again, using Young inequality assuming \\
$\varepsilon=\frac{1}{2\mu }\{N^{1-\alpha}E_{\rho, 2-\alpha}^{-\gamma}(\omega (t_{N})^{\rho})-(N-1)^{1-\alpha}E_{\rho, 2-\alpha}^{-\gamma}(\omega (t_{N-1})^{\rho})\}-\frac{C_0}{2}>0$,

\begin{align}
	& \sum_{k=1}^N  \|\delta_x \epsilon^k\|_h^2\nonumber\\
    & \leq \left(\frac{C_0}{2}- \frac{1}{2\mu }  \{N^{1-\alpha}E_{\rho, 2-\alpha}^{-\gamma}(\omega (t_{N})^{\rho})-(N-1)^{1-\alpha}E_{\rho, 2-\alpha}^{-\gamma}(\omega (t_{N-1})^{\rho})\}\right) \sum_{k=1}^{N} \|\epsilon^k\|_h^2\nonumber \\
	& + \left[\frac{1}{2\mu}\{N^{1-\alpha}E_{\rho, 2-\alpha}^{-\gamma}(\omega (t_{N})^{\rho})-(N-1)^{1-\alpha}E_{\rho, 2-\alpha}^{-\gamma}(\omega (t_{N-1})^{\rho})\}-\frac{C_{0}}{2}\right] \sum_{k=1}^{N} \|\epsilon^k\|_h^2 \nonumber \\
	& + \frac{\sum_{k=1}^{N}\left( \| {(\gamma_1)}^k\|_h^2 +\|{(\gamma_2)}^k\|_h^2 + \|{(\gamma_3)}^k\|_h^2\right) }{2[\frac{1}{2} \{N^{1-\alpha}E_{\rho, 2-\alpha}^{-\gamma}(\omega (t_{N})^{\rho})-(N-1)^{1-\alpha}E_{\rho, 2-\alpha}^{-\gamma}(\omega (t_{N-1})^{\rho})\}-C_{0}]},\nonumber\\
	& \leq C \sum_{k=1}^{N}\left( \| {(\gamma_1)}^k\|_h^2+\|{(\gamma_2)}^k\|_h^2 +\|{(\gamma_3)}^k\|_h^2\right)~,\nonumber
\end{align}
 
\begin{equation*}
	\Rightarrow ~~~~~\tau \sum_{k=1}^N \|\delta_x \epsilon^k\|_h^2 \leq C (\tau^{2-\alpha}-h^2)^2 \quad (\text{using Eq.(\ref{Eq2.29})}).
\end{equation*}

\noindent We can summarise the above result as a following theorem:

 \begin{theorem}
     The finite difference scheme \eqref{main-scheme1} is convergent in the following sense
     	\begin{equation}
		\tau \sum_{k=1}^N \|\delta_x \epsilon^k\|_h^2 \leq C (\tau^{2-\alpha}-h^2)^2,
	\end{equation}
 and the convergence rate $(2-\alpha)$ in time direction and $2$ in space direction.
 \end{theorem}

\section{Numerical Examples}\label{Sec:3}

Before conducting numerical examples we introduce Newton's iterative method (see \cite{qiu2019implicit,singh2023fourth}) to deal with the non-linear term in the numerical scheme \eqref{main-scheme1}.

\subsection{Newton's iterative method}\label{subsection:2.1}

Let us assume $ \psi(u)=uu_x $, and hence the right hand side of \eqref{space-app3} can be denoted as $\psi(u)^k_i$. By rewriting \eqref{main-scheme1}, we get
\begin{align}
	\frac{1}{\mu}[a_1u_i^k-\sum_{j=1}^{k-1}(a_{k-j}-a_{k-j+1})u_i^j & -a_ku_i^0]+ \psi(u)^k_i = \delta^2_x u_i^k+f_i^k,\nonumber\\
	a_1u_i^k -\mu \delta^2_x u_i^k + \mu\psi(u)^k_i = \mu f_i^k & + \sum_{j=1}^{k-1}(a_{k-j} - a_{k-j+1})u_i^j - a_ku_i^0, \nonumber\\
	\implies a_1u_i^k -\mu \delta^2_x u_i^k & + \mu\psi(u)^k_i = H_i^k,\label{Eq2.10}
\end{align}
where 
\begin{equation}\label{Eq2.11}
	H_i^k =\mu f_i^k + \sum_{j=1}^{k-1}(a_{k-j} - a_{k-j+1})u_i^j - a_ku_i^0.
\end{equation}
By applying the Newton's iteration method to $ \psi(u) $ at $ (s+1)^{th} $ level, $s>0$, we have
\begin{align}\label{newton-formula}
	\psi(u^{(s+1)})= \psi(u^{(s)}) + (u^{(s+1)}-u^{(s)})\psi'(u^{(s)}).
\end{align}
Whence Eq. (\ref{Eq2.10}) at $ (s+1)^{th} $ level becomes
\begin{align*}
	a_1(u^{(s+1)})^k_i - \mu \delta^2_x (u^{(s+1)})^k_i + \mu\psi(u^{(s+1)})^k_i = H^k_i.
\end{align*}
Now, after applying the formula \eqref{newton-formula} we get the following equation
\begin{align}
    a_1(u^{(s+1)})^k_i - \mu \delta^2_x (u^{(s+1)})^k_i &+ \mu (u^{(s+1)})^k_i\psi'(u^{(s)})^k_i \nonumber\\&= \mu[(u^{(s)})^k_i\psi'(u^{(s)})^k_i - \psi(u^{(s)})^k_i] + H^k_i,\label{Eq2.12}
\end{align}
and here,
\begin{align}
	\delta^2_x (u^{(s+1)})^k_i&= \frac{1}{h^2}  [(u^{(s+1)})^k_{i+1} -2(u^{(s+1)})^k_i+(u^{(s+1)})^k_{i-1}],\label{Eq.16}\\
	\psi(u^{(s)})^k_i&= \frac{1}{6h}  [(u^{(s)})^k_{i+1} +(u^{(s)})^k_i+(u^{(s)})^k_{i-1}][(u^{(s)})^k_{i+1}-(u^{(s)})^k_{i-1}],\\
	\psi'(u^{(s)})^k_i &=\frac{1}{2h}  [(u^{(s)})^k_{i+1} -(u^{(s)})^k_{i-1}].\label{Eq.18}
\end{align}
So after putting the values \eqref{Eq.16}-\eqref{Eq.18} and doing some simplification, the vector representation of Eq. (\ref{Eq2.12}) becomes
\begin{align}\label{Eq2.24}
	\left[a_1 I - \frac{\mu}{h^2}E_1 + \frac{\mu}{2h}\text{diag}(E_2.(U^{(s)})^k)\right] & (U^{(s+1)})^k \nonumber\\
   = \frac{\mu}{2h}\text{diag}(E_2.(U^{(s)})^k) (U^{(s)})^k & - \frac{\mu}{6h}\text{diag}(E_3.(U^{(s)})^k) E_2(U^{(s)})^k + H^k,
\end{align}
where $ I $ is identity matrix of order $ (M-1 \times M-1) $;
\begin{align*}
    H^k=\mu G^k+ \sum_{j=1}^{k-1}(a_{k-j} - a_{k-j+1})U^j - a_kU^0;  
\end{align*}
\begin{align*}
	E_1 &= \begin{bmatrix}
		-2    & 1       & 0      &     &  \\
		1 	  & -2  	& 1     &      & \\
		& \ddots   & \ddots &\ddots & \\
		& 		& 1		 &-2		 & 1\\
        &         & 0      & 1    & -2
	\end{bmatrix}_{M-1 \times M-1},
	~~~~E_2 = \begin{bmatrix}
		0    & 1       & 0      &     &  \\
		-1 	  & 0  	& 1     &      & \\
		& \ddots   & \ddots &\ddots & \\
        & 		&- 1		 &0		 & 1\\
		&         & 0      & -1    & 0
	\end{bmatrix}_{M-1 \times M-1},
\end{align*}
\begin{align*}
	E_3 &= \begin{bmatrix}
	1    & 1       & 0      &     &  \\
	1 	  & 1  	& 1     &      & \\		
    & \ddots   & \ddots &\ddots & \\
	& 		& 1		 &1		 & 1\\
	&         & 0      & 1    & 1
\end{bmatrix}_{M-1 \times M-1};
~~~~U^k = \begin{bmatrix}
	   u_{1}^k \\
	   u_{2}^k \\
	   \cdots \\
	   u_{M-1}^k
	   \end{bmatrix},
~~~~
F^k = \begin{bmatrix}
		f_{1}^k    \\
		f_{2}^k 	 \\
		\cdots \\
		f_{M-1}^k
	\end{bmatrix}.
\end{align*}
The operator $ \text{diag}(.)$ transforms the vector into a diagonal matrix. 
We substitute $  B:=\left[a_1 I- \frac{\mu}{h^2}E_1 + \frac{\mu}{2h}E_2.(U^{(s)})^k\right]  $ in (\ref{Eq2.24}). The matrix $B$ is invertible, hence we obtain the final scheme as
	\begin{align}
		(U^{(s+1)})^k  = B^{-1} \left[  \frac{\mu}{2h} \right.  \text{diag} & (E_2.(U^{(s)})^k) (U^{(s)})^k  \nonumber \\
         & - \left. \frac{\mu}{6h}\text{diag}(E_3.(U^{(s)})^k) E_2(U^{(s)})^k + H^k \right].
	\end{align}

Now, we conduct some numerical experiments in support of the presented numerical scheme. We calculate the maximum absolute errors ($ \Xi $) and convergence order ($\Theta$) of the proposed scheme by the following formulas considering $ u^k_i $ as numerical solution and  $\mathcal{U}(x,t)$ as the exact solution
\begin{align}
	\Xi_j  := \max_{\substack{{0< i< N}\\ {0< k< M}}} & \vert \mathcal{U}(x_i,t_k)-u^k_i\vert, ~~~~(\text{at}~ j^{\text{th}} ~\text{iteration,})\\
	\Theta_{j+1} := & \ \log_2\left(\frac{\Xi_j}{\Xi_{j+1}}\right).
\end{align}
%
%
\begin{example}\label{Ex1} Consider the problem (\ref{Main.eq})--(\ref{eq2}) with 
$  u_0(x) = 0$ and\\ $f(x,t) = 2sin(\pi x)t^{(2-\alpha)}E_{\rho, 3-\alpha}^{-\gamma}(\omega t^{\rho})+\frac{\pi}{2} t^4 \sin(2\pi x)+ \pi^2 t^2 \sin(\pi x) $ for $(x,t) \in (0,1) \times (0,1)$.
\end{example}
The exact solution of (\ref{Main.eq})--(\ref{eq2}) with the above conditions is $ \ \mathcal{U}(x,t)=t^2 \sin (\pi x) $. To evaluate the numerical solution of this example, we choose Maximum Step = 500 (see Appendix \ref{appA}). We have two tables to show the maximum error ($ \Xi $) and convergence order($ \Theta $) of the scheme. Table  \ref{tab1} is in space direction where iterative accuracy  ($ ItAcc $) is E-05, and Table  \ref{tab2} is in temporal direction with  $ ItAcc $ = E-08. In both tables, we found that the order of convergence is nearly two as the convergence order of the presented numerical scheme is  $ (\tau^{2-\alpha}+h^2) $.
    
\begin{table}[htbp]
	\caption{The error $ \Xi $, order $ \Theta $, computing time, and number of iterations for Example \ref{Ex1} with $N=2^{14}$,  ItAcc = 1E-05,  MaxStep = 500.}\label{tab1}
	\begin{tabular}{cccccc}
		\hline
		$\alpha$                & $M$   & $ \Xi $         & $ \Theta $      & $ \sim $Time & \begin{tabular}[c]{@{}c@{}}No. of\\  Iterations\end{tabular} \\ \hline
		\multirow{4}{*}{0.2} & $2^{6}$ & 2.02791E-04 &         & 123  & 60168 \\
		& $2^{7}$ & 5.25900E-05 & 1.94713 & 942  & 60168 \\
		& $2^{8}$ & 1.35124E-05 & 1.96051 & 875  & 60167 \\
		& $2^{9}$ & 3.75316E-06 & 1.84811 & 2422 & 60167 \\ \hline
		\multirow{4}{*}{0.4} & $2^{6}$ & 1.19380E-04 &         & 1467 & 60168 \\
		& $2^{7}$ & 5.21271E-05 & 1.19546 & 1417 & 60168 \\
		& $2^{8}$ & 1.31737E-05 & 1.98436 & 1696 & 60167 \\
		& $2^{9}$ & 3.43717E-06 & 1.93837 & 2327 & 60167 \\\hline
		\multirow{4}{*}{0.6} & $2^{6}$ & 2.04683E-04 &         & 299  & 31426 \\
		& $2^{7}$ & 5.12785E-05 & 1.99696 & 306  & 31426 \\
		& $2^{8}$ & 1.29318E-05 & 1.98742 & 438  & 31426 \\
		& $2^{9}$ & 3.34602E-06 & 1.95041 & 5661 & 31426 \\ \hline
		\multirow{4}{*}{0.8} & $2^{6}$ & 2.00003E-04 &         & 199  & 31426 \\
		& $2^{7}$ & 5.04469E-05 & 1.98718 & 224  & 31426 \\
		& $2^{8}$ & 1.30657E-05 & 1.94898 & 356  & 31426 \\
		& $2^{9}$ & 3.72789E-06 & 1.80936 & 1153 & 31426 \\ \hline
	\end{tabular}
\end{table}	
	
	\begin{table}[htbp]
			\caption{ The error $ \Xi $, order $ \Theta $, computing time, and number of iterations for Example \ref{Ex1} with $M=2^{12},$ $ItAcc =1E-08$, $MaxStep = 500$.}
			\label{tab2}
		\begin{tabular}{cccccc}
			\toprule
			$\alpha$                & $N$   & $ \Xi $         & $ \Theta $      & $ \sim $Time & \begin{tabular}[c]{@{}c@{}}No. of\\  Iterations\end{tabular} \\ \midrule
			\multirow{4}{*}{0.2} & $2^3$ & 3.86872E-04 &         & 208   & 51  \\
			& $2^4$ & 1.19012E-04 & 1.70075 & 377   & 98  \\
			& $2^5$ & 3.61465E-05 & 1.71917 & 665   & 182 \\
			& $2^6$ & 1.08839E-05 & 1.73165 & 50906 & 346 \\ \midrule
			\multirow{4}{*}{0.4} & $2^3$ & 1.28180E-03 &         & 212   & 51  \\
			& $2^4$ & 4.35384E-04 & 1.55781 & 384   & 94  \\
			& $2^5$ & 1.46489E-04 & 1.57149 & 760   & 179 \\
			& $2^6$ & 4.89727E-05 & 1.58074 & 1205  & 336 \\ \midrule
			\multirow{4}{*}{0.6} & $2^3$ & 3.16257E-03 &         & 147   & 49  \\
			& $2^4$ & 1.20729E-03 & 1.38932 & 297   & 93  \\
			& $2^5$ & 4.58333E-04 & 1.39731 & 562   & 171 \\
			& $2^6$ & 1.73444E-04 & 1.40192 & 955   & 318 \\ \midrule
			\multirow{4}{*}{0.8} & $2^3$ & 7.00580E-03 &         & 152   & 49  \\
			& $2^4$ & 3.03864E-03 & 1.20513 & 273   & 88  \\
			& $2^5$ & 1.31466E-03 & 1.20874 & 520   & 162 \\
			& $2^6$ & 5.68242E-04 & 1.21011 & 1047  & 298 \\ \botrule
		\end{tabular}
	\end{table}
	
	\begin{example}\label{Ex2} Consider Eq.(\ref{Main.eq})-(\ref{eq2}) with $ 
		u_0(x) = 0$. Let $ L=1=T=1, ~~f(x,t) = 120 x(x-1)t^{5-\alpha} E_{\rho,6-\alpha}^{-\gamma}(\omega t^{\rho}) + x(x-1)(2x-1)t^{10} -2t^5 $.
		
	\end{example}
	The exact solution of this example is $\ \mathcal{U}(x,t)=t^5 x(x-1) $. Here, we choose  iterative accuracy  $ ItAcc $ = E-08 for time direction (Table \ref{tab3}) and   $ ItAcc $ = E-06 for space direction (Table \ref{tab4}), and Maximum Step = 500  (see Appendix \ref{appA}). Again, the obtained order of convergence of the scheme is approximately two. 
	
\begin{table}[htbp]
\centering
	\caption{The error $ \Xi $, order $ \Theta $, computing time, and number of iterations for Example \ref{Ex2} with $M=2^{11}$, $MaxStep = 500 $, $ItAcc =1E-08$.}\label{tab3}
	\begin{tabular}{cccccc}
		\toprule
		$\alpha$                & $N$   & $ \Xi $         & $ \Theta $      & $ \sim $Time & \begin{tabular}[c]{@{}c@{}}No. of\\  Iterations\end{tabular} \\ \midrule
		\multirow{4}{*}{0.2} & $2^3$ & 6.63368E-04 &         & 10  & 29  \\
		& $2^4$ & 2.23733E-04 & 1.56803 & 28  & 57  \\
		& $2^5$ & 7.21425E-05 & 1.63286 & 55  & 107 \\
		& $2^6$ & 2.26123E-05 & 1.67374 & 106 & 205 \\ \midrule
		\multirow{4}{*}{0.4} & $2^3$ & 2.28046E-03 &         & 10  & 29  \\
		& $2^4$ & 8.46655E-04 & 1.42948 & 24  & 56  \\
		& $2^5$ & 3.01189E-04 & 1.49110 & 56  & 105 \\
		& $2^6$ & 1.04408E-04 & 1.52843 & 103 & 202 \\ \midrule
		\multirow{4}{*}{0.6} & $2^3$ & 5.65044E-03 &         & 12  & 29  \\
		& $2^4$ & 2.33444E-03 & 1.27529 & 27  & 55  \\
		& $2^5$ & 9.28672E-04 & 1.32984 & 52  & 104 \\
		& $2^6$ & 3.61486E-04 & 1.36123 & 90  & 197 \\ \midrule
		\multirow{4}{*}{0.8} & $2^3$ & 1.22317E-02 &         & 12  & 29  \\
		& $2^4$ & 5.65531E-03 & 1.11294 & 27  & 55  \\
		& $2^5$ & 2.53554E-03 & 1.15731 & 52  & 103 \\
		& $2^6$ & 1.11798E-03 & 1.18139 & 78  & 192 \\ \toprule
	\end{tabular}
\end{table}

	\begin{table}[htbp]
			\caption{ The error $ \Xi $, order $ \Theta $, computing time, and number of iterations for Example \ref{Ex2} with $N=2^{15}$, $ItAcc =1e-6$,   $MaxStep  = 500$.}
			\label{tab4}
			\begin{tabular}{cccccc}
				\toprule
		$\alpha$                & $N$    & $ \Xi $         & $ \Theta $      & $ \sim $Time & \begin{tabular}[c]{@{}c@{}}No. of\\  Iterations\end{tabular} \\ \midrule
		\multirow{4}{*}{0.2} & $2^6$ & 4.98831E-06 &         & 911  & 52351 \\
		& $2^7$ & 1.25116E-06 & 1.99529 & 854  & 52351 \\
		& $2^8$ & 3.20518E-07 & 1.96479 & 1052 & 52351 \\
		& $2^9$ & 8.83646E-07 & 1.85886 & 5379 & 52351 \\ \midrule
		\multirow{4}{*}{0.4} & $2^6$ & 4.87048E-06 &         & 756  & 52351 \\
		& $2^7$ & 1.22365E-06 & 1.99287 & 915  & 52351 \\
		& $2^8$ & 3.11097E-07 & 1.97575 & 2032 & 52351 \\
		& $2^9$ & 8.31154E-08 & 1.90418 & 2525 & 52351 \\ \midrule
		\multirow{4}{*}{0.6} & $2^6$ & 4.71577E-06 &         & 734  & 52351 \\
		& $2^7$ & 1.20694E-06 & 1.96613 & 902  & 52351 \\
		& $2^8$ & 3.29704E-07 & 1.87211 & 3319 & 52351 \\
		& $2^9$ & 1.13032E-07 & 1.54443 & 2720 & 52351 \\ \midrule
		\multirow{4}{*}{0.8} & $2^6$ & 4.76969E-05 &         & 1632 & 52351 \\
		& $2^7$ & 1.21947E-06 & 1.96764 & 1565 & 52351 \\
		& $2^8$ & 3.26153E-07 & 1.90264 & 1775 & 52351 \\
		& $2^9$ & 8.46157E-08 & 1.94655 & 3503 & 52351 \\ \toprule
		\end{tabular}
	\end{table}


\section{Conclusions}\label{sec:4}
In conclusion, we have effectively solved the fractional Burgers equation using an implicit difference scheme that incorporates the Caputo-Prabhakar derivative. Our analysis establishes the stability and convergence of the proposed numerical scheme, demonstrating a convergence order of $ (\tau^{2-\alpha}+h^2) $. The method described in this paper achieves a better order of convergence in time direction, as compared to the paper \cite{yadav2020numerical}.  Moreover, the numerical experiments conducted support our theoretical findings, illustrating the robustness and effectiveness of our approach.
{While this work establishes a solid foundation for the numerical analysis of the fractional Burgers equation, we have not considered the initial singularity of the solution, which is commonly observed in many fractional PDEs \cite{jin2016two,stynes2022survey}. Addressing this aspect will be the focus of our future study. However, the approach discussed here still provides a strong basis for future advancements and remains applicable in cases where the initial singularity has minimal impact or can be addressed through alternative methods in subsequent analysis.}

\backmatter

\bmhead{Acknowledgements}
Authors thanks W. Qiu for the help when the work was at initial level. 

\section*{Declarations}

\begin{itemize}
\item Funding: DS is supported by UGC under SRF scheme. SY is supported by MUR (PRIN2022 research grant n. 202292JW3F). This work is partially supported by the ERCIM ‘Alain Bensoussan’ Fellowship programme.
\item Conflict of interest/Competing interests: The authors have no competing interests to declare.
\item Ethics approval and consent to participate: Not applicable.
\item Consent for publication: Not applicable.
\item Data availability: Not applicable.
\item Materials availability: Not applicable.
\item Code availability: Not applicable. 
\item Author contribution: All the authors contributed equally to this work.
\end{itemize}


\begin{appendices}

\section{Algorithm}\label{appA}

The algorithm to compute the numerical solution with the help of iteration methods is given in the article presented by Qiu et al. \cite{qiu2019implicit}. In this appendix, we provide the same algorithm for the ease of the reader.
For computational work, it is required to set maximum iteration step $ MaxStep $, and iterative accuracy $ ItAcc $. The algorithm used in the paper is as follows:
		
\centering{
    \begin{algorithm}[h]{\allowdisplaybreaks}
		\SetAlgoLined
		\KwIn{$U_0, MaxStep, N, ItAcc$}
		\KwOut{$U_n(1\leq n \leq N)$ }
		initialization: $U_1^{(0)}=U_0$ \;
		\For{n=1:N}{
			Calculate $H_i^k$ from the Eq. (\ref{Eq2.10})\;
			\While{(l)}{
				s=s+1\;
				Obtain $U_n^{(s)}$ by the iterative scheme ()\;
				\eIf{max[abs($U_n^{(s)}$-$U_n^{(s-1)}$)]$ < $ItAcc}{
					$U_n-U_n^{(s)}$\;
					break\;
				}{
				\If{$s\geq $ MaxStep}{
					Print: ``It is not convergent within the given number of steps"\;
							return\;
					}
				}
			}
			$U_{n+1}^{(0)}=U_n$ \;
		}
	return $U_n(1\leq n \leq N)$ \;
	\caption{To compute numerical solution}
    \end{algorithm}
}

\end{appendices}


\end{document}